\documentclass[11pt,reqno]{amsart}

\usepackage[utf8]{inputenc}
\usepackage[T1]{fontenc}
\usepackage[english]{babel}
\usepackage{csquotes}

\usepackage[margin=3cm]{geometry}

\usepackage{xcolor}
\usepackage[colorlinks=true, linkcolor=black, citecolor=black, urlcolor=black]{hyperref} 
\usepackage[nameinlink]{cleveref}

\usepackage[onehalfspacing]{setspace}

\numberwithin{equation}{section} 

\usepackage{paralist}
\usepackage{enumitem}
\setlist[enumerate,1]{label = (\alph*),leftmargin=*, topsep=1mm, itemsep=1mm}
\setlist[enumerate,2]{label = (\roman*),leftmargin=*, topsep=1mm, itemsep=1mm}
\setlist[itemize,1]{leftmargin=*, topsep=1mm, itemsep=1mm}

\usepackage{times}
\usepackage{txfonts}
\usepackage{dsfont}
\renewcommand{\mathbb}{\mathds}

\usepackage{amsmath,amssymb}
\usepackage{mathdots}
\usepackage{extpfeil}
\usepackage{extarrows}
\usepackage{oubraces}
\usepackage{array}

\usepackage{tikz}
\usetikzlibrary{cd,arrows,decorations.markings}

\tikzset{
	marrow/.style={decoration={markings,mark=at position 0.75 with {\arrow{#1}}}, postaction=decorate}
}

\tikzcdset{
	arrow style=tikz,
	diagrams={>={Straight Barb[scale=0.9, width=5pt]}} }

\usepackage[backend=biber,style=alphabetic,maxbibnames=4,doi=false,isbn=false,url=false]{biblatex}
\bibliography{trirep.bib}

\theoremstyle{plain}
\newtheorem{thm}{Theorem}[section]
\newtheorem{prp}[thm]{Proposition}

\newtheorem{lem}[thm]{Lemma}

\newtheorem*{prp*}{Proposition}
\newtheorem*{thm*}{Theorem}

\theoremstyle{definition}
\newtheorem{dfn}[thm]{Definition}
\newtheorem{ntn}[thm]{Notation}

\newtheorem{cnv}[thm]{Convention}

\theoremstyle{remark}
\newtheorem{rmk}[thm]{Remark}

\Crefname{subsection}{Subsection}{Subsections}
\Crefname{thm}{Theorem}{Theorems}
\Crefname{thmA}{Theorem}{Theorems}
\Crefname{prp}{Proposition}{Propositions}
\Crefname{cor}{Corollary}{Corollaries}
\Crefname{lem}{Lemma}{Lemmas}
\Crefname{cnj}{Conjecture}{Conjectures}
\Crefname{dfn}{Definition}{Definitions}
\Crefname{ntn}{Notation}{Notations}
\Crefname{con}{Construction}{Constructions}
\Crefname{asn}{Assumption}{Assumptions}
\Crefname{cnv}{Convention}{Conventions}
\Crefname{rmk}{Remark}{Remarks}
\crefname{rmk}{Remark}{Remarks}
\Crefname{exa}{Example}{Examples}

\newcommand{\I}{\mathcal{I}}

\newcommand{\M}{\mathcal{M}}
\newcommand{\N}{\mathcal{N}}
\newcommand{\cS}{\mathcal{S}}
\newcommand{\T}{\mathcal{T}}
\newcommand{\U}{\mathcal{U}}
\newcommand{\V}{\mathcal{V}}
\newcommand{\W}{\mathcal{W}}

\newcommand{\sfS}{\mathsf{S}}
\newcommand{\sfM}{\mathsf{M}}
\newcommand{\sfN}{\mathsf{N}}
\newcommand{\sfU}{\mathsf{U}}
\newcommand{\sfV}{\mathsf{V}}

\newcommand{\op}[1]{{#1}^\textup{op}}

\DeclareMathOperator{\id}{id}
\DeclareMathOperator{\ob}{ob}
\DeclareMathOperator{\ev}{ev}

\title[Representable functors and semiorthogonal decompositions]{Representable triangulated functors\\in terms of semiorthogonal decompositions}

\author[J.~Frank]{Jonas Frank}
\address{\linebreak
	Jonas Frank\\
	Department of Mathematics, University of Kaiserslautern-Landau\\ 
	67663 Kaiserslautern\\
	Germany
}
\email{\href{jfrank@rptu.de}{jfrank@rptu.de}}

\author[M.~Schulze]{Mathias Schulze}
\address{\linebreak
	Mathias Schulze\\
	Department of Mathematics, University of Kaiserslautern-Landau\\ 
	67663 Kaiserslautern\\
	Germany
}
\email{\href{mschulze@rptu.de}{mschulze@rptu.de}} 

\subjclass[2020]{Primary 18A22, 18E40; Secondary 18G80, 18D20}

\keywords{Representable functor, triangulated category, semiorthogonal decomposition, enriched category, monoidal category.}

\begin{document}

\begin{abstract}
	A theorem of Bondal and Kapranov lifts representations of cohomological functors from semiorthogonal decompositions of triangulated categories. We present a version of this result for triangulated functors. To this end, we introduce suitable terminology in enriched category theory and transfer the original proof.
\end{abstract}

\maketitle

\tableofcontents

\section{Introduction}

Bondal and Kapranov {\cite{BK89}} develop a categorical framework to study triangulated categories by means of filtrations. It is motivated by applications to the derived category of coherent sheaves on projective space, where such filtrations can be generated by an exceptional set or can be induced from a stratification in terms of the cohomology support.\\
Central to their approach is the notion of admissibility of (triangulated) subcategories or, more generally, of filtrations by (triangulated) subcategories: An admissible subcategory together with the respective quotient category forms a semiorthogonal decomposition of the ambient category, the successive quotients of an admissible filtration a semiorthogonal generating sequence.\\
In close relation to admissibility, the authors introduce the important notion of a Serre functor on a linear category (over a field). It generalizes the Serre--Grothendieck duality in algebraic geometry and the Nakayama functor in the representation theory of finite dimensional algebras. Serre functors are essentially unique, given on objects by representing the respective cohomological dualized hom-functor and on morphisms via the Yoneda lemma.\\
A collection of Serre functors on the successive quotients of a strongly admissible filtration lifts to a Serre functor on the ambient category. In the base case, the filtration has length two, see {\cite[Thm.~3.8]{BK89}}:

\begin{prp*}[Bondal--Kapranov]
	Let $\M$ be a triangulated category, linear over a field, with finite dimensional hom-vector spaces. Suppose that $(\U, \V)$ is a semiorthogonal decomposition of $\M$. Then $\M$ admits a Serre functor if both $\U$ and $\V$ do so.
\end{prp*}

The proof relies on a procedure for lifting representations of cohomological functors from semiorthogonal decompositions, see {\cite[Thm.~2.10]{BK89}}, which we adapt to triangulated functors. To this end, the hom-vector spaces need to be replaced by objects of another triangulated category $\sfS$. Assuming that $\sfS$ is additive monoidal, this is realized by an $\sfS$-(enriched) category $\M$ with triangulated underlying category $\sfM$, whose preadditive structure comes from $\sfS$, see \Cref{dfn: underlying-cat}. Further assuming that $\sfS$ is closed, $\sfS$ admits an internal hom. We do not assume that the tensor product or the internal hom of $\sfS$ are triangle functors in their components. The internal homs turn $\sfS$ into an $\sfS$-category $\cS$, the \emph{self-enrichment}. We call an $\sfS$-functor $\M \to \cS$ \emph{triangulated} if the underlying ordinary functor $\sfM \to \sfS$ is so. The hom-objects $\M(U, V) \in \sfS$ of $\M$ give rise to such $\sfS$-functors $\M(U, -)$ and $\M(-, V)$, \emph{(co)represented} by $U, V \in \M$. Our definition of \emph{$\sfS$-triangulated category}, see \Cref{dfn: enriched-triang}, requires that these be triangulated.
Our enriched notion of an $\sfS$-semiorthogonal decomposition $(\U, \V)$ of $\M$ imposes the enriched semiorthogonality condition $\M(\U, \V)=0$ on the $\sfS$-subcategories $\U$ and $\V$ of $\M$ retaining the decomposition $\sfM = \sfU \ast \sfV$ of the underlying categories, see \Cref{dfn: enriched-sub,dfn: SOD}.

\begin{thm*}
	Let $\sfS$ be a symmetric closed additive monoidal and triangulated category. Consider an $\sfS$-triangulated category $\M$ with semiorthogonal decompositions $(\W, \U)$ and $(\U, \V)$. Then a triangulated $\sfS$-functor $\tau \colon \M \to \cS$ is $\sfS$-representable if $\tau \vert_{\U}$ and $\tau \vert_{\V}$ are so.
\end{thm*}

In a similar spirit, Steen and Stevenson {\cite{SS17}} establish a enriched version of Brown representability. In order to adapt Neeman's classical proof {\cite{Nee96}}, they require the existence of copowers in $\M$. As opposed to our setup, their hypotheses include that the tensor product and the internal hom of $\sfS$ are triangle functors in each component.

\section{Monoidal and enriched categories}

In this and the next section, we collect the essentials on monoidal and enriched categories. In particular, we recall the weak form of the enriched Yoneda lemma and discuss implications of the bifunctoriality of hom-objects. Our main reference on this topic is the reprint {\cite{Kel05}} of Kelly's book {\cite{Kel82}}.

\begin{ntn} \
	\begin{enumerate}
		\item We denote the data of a ((pre)additive) monoidal category $\sfS$ as follows: $I \in \sfS$ is the \emph{unit object}, $\otimes \colon \sfS \times \sfS \to \sfS$ the (biadditive) \emph{tensor} bifunctor, and
		\[a=a_{X, Y, Z} \colon (X \otimes Y) \otimes Z \longrightarrow X \otimes (Y \otimes Z), \hspace{2.5mm} l=l_X \colon I \otimes X \to X, \hspace{2.5mm} \textup{ and } \hspace{2.5mm} r=r_X \colon X \otimes I \to X\]
		are the \emph{associativity}, \emph{left} and \emph{right unit} isomorphisms, respectively, natural in $X, Y, Z \in \sfS$. For symmetric $\sfS$, there are \emph{symmetry} isomorphisms \[c=c_{X, Y} = c_{Y, X}^{-1} \colon X \otimes Y \to Y \otimes X,\] natural in $X, Y \in \sfS$. Compatibility of these components is due to certain \emph{coherence axioms}, see {\cite[(1.1), (1.2), (1.15), (1.16)]{Kel05}}.
		
		\item We denote the data of an $\sfS$-(enriched) category $\M$ as follows: $\ob \M$ is the set of \emph{objects}, $\M(U, V) \in \sfS$ the \emph{hom-objects}, and
		\[M=M_{U, V, W} \colon \M(V, W) \otimes \M(U, V) \longrightarrow \M(U, W) \hspace{2.5mm} \textup{ and } \hspace{2.5mm} j=j_U \colon I \to \M(U, U)\]
		are the \emph{composition} and the \emph{identity} morphisms, for $U, V, W \in \ob \M$. These satisfy \emph{associativity} and \emph{unit axioms}, ensuring their compatibility (with the data defining $\sfS$), see {\cite[(1.3), (1.4)]{Kel05}}. We write $U \in \M$ for $U \in \ob \M$.
		
		\item For symmetric $\sfS$, let $\op \M$ denote the $\sfS$-category with $\ob \op \M := \ob \M$, $\op \M(U, V) := \M(V, U)$, for $U, V \in \op \M$, composition $M \circ c$, and the same identity morphisms as $\M$.
	\end{enumerate}
\end{ntn}

\begin{rmk} \label{rmk: r=l-I}
	Using the coherence axioms and the naturality of the isomorphism $r$, one can show that $r_I=l_I$, see {\cite[Cor.~2.2.5]{Eti+15}}. This means that $c_{I, I} = \id_{I \otimes I}$ in the symmetric case.
\end{rmk}

\begin{dfn} \label{dfn: underlying-cat}
	For a (preadditive) monoidal category $\sfS$, the (preadditive) category $\sfM$ \textbf{underlying} the $\sfS$-category $\M$ has objects $\ob \sfM := \ob \M$ and morphisms $$\sfM(U, V) := \sfS(I, \M(U, V)),$$ for $U, V \in \M$. The (biadditive) composition $gf := g \circ f := g \circ_\sfM f \colon U \to W$ is given by 
	\begin{center}
		\begin{tikzcd}
			I \ar[r, "l^{-1} \, = \, r^{-1}"] & I \otimes I \ar[r, "g \otimes f"] & \M(V, W) \otimes \M(U, V) \ar[r, "M"] & \M(U, W),
		\end{tikzcd}
	\end{center}
	for two morphisms $f \colon U \to V$ and $g \colon V \to W$ in $\sfM$, see \Cref{rmk: r=l-I}. Identity morphisms are then given by $\id_U := j_U$, for $U \in \M$. Note that $\op\sfM$ underlies $\op \M$, see {\cite[\S1.4]{Kel05}}.
\end{dfn}

\begin{dfn}
	Let $\M$ and $\N$ be $\sfS$-categories.
	\begin{enumerate}
		\item An \textbf{$\sfS$-functor} $\tau \colon \M \to \N$, consists of a map $\ob \M \to \ob \N$, $U \mapsto \tau U$, and morphisms
		\[\tau = \tau_{U, V} \colon \M(U, V) \to \N(\tau U, \tau V)\] in $\sfS$, for each $U, V \in \M$, such that
		\begin{gather} \label{diag: dfn-functor}
			\begin{tikzcd}[row sep={17.5mm,between origins}, column sep={35mm,between origins}, ampersand replacement=\&]
				\M(V, W) \otimes \M(U, V) \ar[r, "M"] \ar[d, "\tau \otimes \tau"] \& \M(U, W) \ar[d, "\tau"] \\
				\N(\tau V, \tau W) \otimes \N(\tau U, \tau V) \ar[r, "M"] \& \N(\tau U, \tau W)
			\end{tikzcd}
			\hspace{2.5mm} \textup{and} \hspace{2.5mm}
			\begin{tikzcd}[column sep={17.5mm,between origins}, row sep={8.75mm,between origins}, ampersand replacement=\&]
				\& \M(U, U) \ar[dd, "\tau"] \\
				I \ar[ru, "j"] \ar[rd, "j"'] \& \\
				\& \N(\tau U, \tau U)
			\end{tikzcd}
		\end{gather}
		commute, for all $U, V, W \in \M$, see {\cite[(1.5), (1.6)]{Kel05}}.
		
		\item An \textbf{$\sfS$-natural transformation} $\eta = (\eta_U)_{U \in \M} \colon \sigma \to \tau$ of $\sfS$-functors $\sigma, \tau \colon \M \to \N$ consists of \textbf{component} morphisms $\eta_U \colon \sigma U \to \tau U$ in $\sfN$, for each $U \in \M$, satisfying the \textbf{$\sfS$-naturality condition}, see {\cite[(1.7)]{Kel05}}. If $\sfS$ is closed, there is an alternative condition, see \Cref{lem: nat-alt}. We denote the collection of such $\sfS$-natural transformations by $\sfS\textup{-Nat}(\sigma, \tau)$.
		
		\item The \textbf{identity} of an $\sfS$-functor $\tau\colon \M \to \N$ is the $\sfS$-natural transformation $\id_\tau \colon \tau \to \tau$ defined by $(\id_\tau)_U := j_{\tau U}$, for each $U \in \M$.
		
		\item The \textbf{(vertical) composite} of two $\sfS$-natural transformations $\eta \colon \sigma \to \tau$ and $\theta \colon \tau \to \upsilon$ of $\sfS$-functors $\sigma, \tau, \nu \colon \M \to \N$ is the $\sfS$-natural transformation $\theta \eta := \theta \circ \eta \colon \sigma \to \upsilon$ defined by $(\theta \eta)_U := \theta_U \circ_\sfN \eta_U$, for each $U \in \M$.
		
		\item An $\sfS$-natural transformation $\eta \colon \sigma \to \tau$ of $\sfS$-functors $\sigma, \tau \colon \M \to \N$ is an \textbf{$\sfS$-natural isomorphism} or \textbf{isomorphism of $\sfS$-functors} if $\theta \eta = \id_\sigma$ and $\eta \theta = \id_\tau$, for some $\sfS$-natural transformation $\theta \colon \tau \to \sigma$. Equivalently, if $\sfS$ is closed, $\eta_U$ is an isomorphism in $\sfN$, for each $U \in \M$, see \Cref{lem: iso-components}.
	\end{enumerate}
\end{dfn}

\begin{dfn} \label{dfn: underlying-functor}
	For a (preadditive) monoidal category $\sfS$, any $\sfS$-functor $\tau \colon \M \to \N$ has an \textbf{underlying} (preadditive) functor $\tau \colon \sfM \to \sfN$, denoted by the same symbol, defined on morphisms $f \in \sfM(U, V)$ by
		\begin{center}
			\begin{tikzcd}
				\sfN(\tau U, \tau V) \ni \tau f := \sfS(I, \tau_{U, V})(f) \colon I \ar[r, "f"] & \M(U, V) \ar[r, "\tau_{U, V}"] & \N(\tau U, \tau V).
			\end{tikzcd}
		\end{center}
\end{dfn}

\begin{rmk} \label{rmk: underlying-nat}
	Any $\sfS$-natural transformation is an (ordinary) natural transformation of the underlying functors, given by the exact same data. In general, (ordinary) naturality is weaker than $\sfS$-naturality, see {\cite[\S1.3]{Kel05}}.
\end{rmk}

\begin{rmk}
	A more conceptual approach defines the underlying category of an $\sfS$-category $\M$ as the category of $\sfS$-functors $\I \to \M$, where $\I$ is a certain \emph{unit $\sfS$-category}, see {\cite[\S1.3]{Kel05}}. Accordingly, morphisms in $\sfM$ are $\sfS$-natural transformations between such functors. The functor underlying an $\sfS$-functor $\tau \colon \M \to \N$ then operates on morphisms by whiskering from the left, which gives the same meaning to the notation $\tau f$ as in \Cref{dfn: underlying-functor}.
\end{rmk}

\begin{dfn} \label{dfn: int-hom}
	An (additive) monoidal category $\sfS$ is said to be \textbf{closed} if, for each $Y \in \sfS$, the (additive) functor $- \otimes Y \colon \sfS \to \sfS$ has an (additive) right adjoint $\left[Y , -\right] \colon \sfS \to \sfS$, the \textbf{internal hom}.\footnote{Note that adjoints between additive categories are additive.} The adjunction (group) isomorphisms
	\begin{gather}\label{eqn: adjunction}
		\pi = \pi_{X, Y, Z} \colon \sfS(X \otimes Y, Z) \xlongrightarrow{\cong} \sfS(X, [Y, Z]), \; f \longmapsto \tilde f,
	\end{gather}
	are natural in $X, Z \in \sfS$. We denote by $\ev = \ev_{Y, Z} \colon [Y, Z] \otimes Y \to Z$ its counit, the \textbf{evaluation}.
\end{dfn}

\begin{rmk} \label{rmk: rules} Let $\sfS$ be a closed (additive) monoidal category.
	\begin{enumerate}
		\item \label{rmk: rules-co-unit} Note that $f$ and $\tilde f$ in \eqref{eqn: adjunction} are related by the commutative squares
		\begin{center}
			\begin{tikzcd}[sep={17.5mm,between origins}]
				X \otimes Y \ar[r, "f"] \ar[d, "\tilde f \otimes Y"'] & Z \\
				\left[Y, Z \right] \otimes Y \ar[ru, "\ev_{Y, Z}"']
			\end{tikzcd}
			\hspace{2.5mm} and \hspace{2.5mm}
			\begin{tikzcd}[sep={17.5mm,between origins}]
				X \ar[r, "\tilde f"] & \left[Y, Z\right] \ar[d, "{[}Y{,}\,f{]}"] \\
				& \left[Y, X \otimes Y \right] \ar[lu, "d_{X, Y}"],
			\end{tikzcd}
		\end{center}
		where $d = d_{X, Y} \colon \left[Y, X \otimes Y \right] \to X$ is the unit of the adjunction, see {\cite[\S1.9]{Kel05}}.
		
		\item Consider an object $Y \in \sfS$ and morphisms $f \colon X' \to X$ and $h \colon Z \to Z'$ in $\sfS$. Since the isomorphism $\pi_{X, Y, Z}$ from \eqref{eqn: adjunction} is natural in $X$ and $Z$, we obtain a commutative diagram
		\begin{center}
			\begin{tikzcd}[row sep={17.5mm,between origins}, column sep={35mm,between origins}]
				\sfS(X \otimes Y, Z) \ar[r, "\pi_{X, Y, Z}", "\cong"'] \ar[d, "\sfS(f \otimes Y{,}\,h)"'] & \sfS(X, [Y, Z]) \ar[d, "\sfS(f{,}\,{[}Y{,}\,h{]})"]\\
				\sfS(X' \otimes Y, Z') \ar[r, "\pi_{X', Y, Z'}", "\cong"'] & \sfS(X', [Y, Z']).
			\end{tikzcd}
		\end{center}
		Given any morphism $g \colon X \otimes Y \to Z$ in $\sfS$, this means that
		\begin{gather}\label{eqn: rule-binat}
			\pi(h \circ g \circ (f \otimes Y)) = [Y, h] \circ \pi(g) \circ f.
		\end{gather}
		
		\item \label{rmk: rules-mid-arg} Let $Z \in \sfS$ and $f \colon Y' \to Y$ be a morphism in $\sfS$. Since $\otimes$ is an (additive) bifunctor, the ordinary Yoneda lemma yields a unique morphism $[f, Z] \colon [Y, Z] \to [Y', Z]$ in $\sfS$ defined by the following commutative square of natural transformations:
		\begin{center}
			\begin{tikzcd}[row sep={17.5mm,between origins}, column sep={35mm,between origins}]
				\sfS(- \otimes Y, Z) \ar[r, "\pi_{-, Y, Z}", "\cong"'] \ar[d, "\sfS(- \otimes f{,}\,Z)"'] & \sfS(-, \left[Y, Z\right]) \ar[d, dashed, "\sfS(-{,}\,{[}f{,}\, Z{]})"]\\
				\sfS(- \otimes Y', Z) \ar[r, "\pi_{-, Y, Z'}", "\cong"'] & \sfS(-, \left[Y', Z\right])
			\end{tikzcd}
		\end{center}
		This gives rise to an (additive) functor $[-, Z] \colon \sfS \to \sfS$. Given $X \in \sfS$ and any morphism $g \colon X \otimes Y \to Z$ in $\sfS$, the commutativity of the above square means that
		\begin{gather}\label{eqn: rule-mid-arg}
			\pi(g \circ (X \otimes f)) = \left[f, Z\right] \circ \pi(g).
		\end{gather}
	\end{enumerate}
\end{rmk}

\begin{prp}[{\cite[Prop.~6.1.7]{Bor94}}] \label{prp: int-hom-bifunctor}
	For a closed monoidal category $\sfS$, there is a bifunctor $[-,-] \colon \op \sfS \times \sfS \to \sfS$, whose partial functors are $[Y, -]$ and $[-, Z]$, for each $Y, Z \in \sfS$, see \Cref{rmk: rules}.\ref{rmk: rules-mid-arg}. \qed
\end{prp}

\begin{ntn}\
	\begin{enumerate}
		\item For $Y \in \sfS$, we write $i_Y := \pi(r_Y) \in \sfS(Y, [I, Y])$, which is a natural isomorphism, see {\cite[\S1.5]{Kel05}}.
		
		\item For $Y \in \sfS$ and a morphism $e \colon I \to X$ in $\sfS$, we call the composite
		\[-(e) = -(e)_{X, Y} \colon [X, Y] \xrightarrow{[e, Y]} [I, Y] \xrightarrow{i^{-1}_{Y}} Y\]
		the \textbf{evaluation} at $e$. It can be written in terms of the evaluation from \Cref{dfn: int-hom} as $-(e) \circ r_{[X, Y]} = \ev_{I, Y} \circ \left([e, Y] \otimes I \right)$, due to \Cref{rmk: rules}.\ref{rmk: rules-co-unit} and the naturality of $r$.
	\end{enumerate}
\end{ntn}

\begin{rmk} \label{rmk: evaluation}
	Let $Y \in \sfS$ and $e \colon I \to X$ be a morphism in $\sfS$.
	\begin{enumerate}
		\item \label{rmk: evaluation-comp} By definition, $-(f \circ e) = -(e) \circ [f, Y]$, for any morphism $f \colon X \to X'$ in $\sfS$.
		
		\item \label{rmk: evaluation-nat} The morphism $-(e) \colon [X, Y] \to Y$ is natural in $Y$ since $[-,-]$ is a bifunctor, see \Cref{prp: int-hom-bifunctor}, and $i^{-1}$ natural: For any morphism $g \colon Y \to Y'$, there is a commutative diagram
		\begin{center}
			\begin{tikzcd}[sep={17.5mm,between origins}]
				\left[X, Y \right] \ar[rd, "{[}e{,}\, Y{]}"] \ar[rrr, "{[}X{,}\, g{]}"] \ar[dd, "-(e)_{X, Y}"'] &&& \left[X, Y'\right] \ar[ld, "{[}e{,}\, Y'{]}"'] \ar[dd, "-(e)_{X, Y'}"] \\
				& \left[I, Y\right] \ar[ld, "i_Y^{-1}"'] \ar[r, "{[}I{,}\, g{]}"] & \left[I, Y'\right] \ar[rd, "i_{Y'}^{-1}"] \\
				Y \ar[rrr, "g"] &&& Y'.
			\end{tikzcd}
		\end{center}
	\end{enumerate}
\end{rmk}

\begin{prp}[{\cite[\S1.6]{Kel05}}]
	Any closed monoidal category $\sfS$ yields an $\sfS$-category $\cS$ with objects $\ob \cS := \ob \sfS$ and hom-objects $\cS(X, Y) := [X, Y]$, for each $X, Y \in \cS$. Under the adjunction \eqref{eqn: adjunction}, the composition $M \colon [Y, Z] \otimes [X, Y] \to [X, Z]$ corresponds to
	\begin{center}
		\begin{tikzcd}
			\left(\left[Y, Z\right] \otimes \left[X, Y\right]\right) \otimes X \ar[r, "a", "\cong"'] & \left[Y, Z\right] \otimes \left( \left[X, Y\right] \otimes X \right) \ar[r, "\id \otimes \ev"] & \left[Y, Z\right] \otimes Y \ar[r, "\ev"] & Z,
		\end{tikzcd}
	\end{center}
	and the identity element $j_X \colon I \to [X, X]$ to $l_X \colon I \otimes X \to X$, for each $X, Y, Z \in \cS$. \qed
\end{prp}

\begin{rmk} \label{rmk: self-enrichment}
	For a closed (additive) monoidal category $\sfS$, the underlying (preadditive) category of the $\sfS$-category $\cS$ is (additively) isomorphic to the (additive) category $\sfS$. Therefore, $\cS$ is referred to as the \textbf{self-enrichment} of $\sfS$. Indeed, both categories have the same objects and there are (group) isomorphisms
	\begin{gather}
		\begin{tikzcd}[ampersand replacement=\&]
			\sfS(X, Y) \ar[r, "\sfS(l{,}\, Y)", "\cong"'] \& \sfS(I \otimes X, Y) \ar[r, "\pi_{I, X, Y}", "\cong"'] \& \sfS(I, [X, Y]),
		\end{tikzcd}
	\end{gather}
	for all $X, Y \in \cS$, compatible with identities and compositions, see {\cite[\S1.6]{Kel05}}.
\end{rmk}

\section{Enriched representability and Yoneda lemma}

In this section, $\sfS$ is a closed (additive) monoidal category with self-enrichment $\cS$.

\begin{cnv}
	For any $\sfS$-functor $\tau \colon \M \to \cS$, we use the identification from \Cref{rmk: self-enrichment} to redefine the target of its (additive) underlying functor to be the category $\sfS$. More explicitly, we replace \[\sfS(I, \tau_{U, V}) \colon \sfM(U, V) = \sfS(I, \M(U, V)) \to \sfS(I, [\tau U, \tau V]), \; f \mapsto \tau f,\] for all $U, V \in \M$, by 
	\begin{center}
		\begin{tikzcd}[ampersand replacement=\&]
			\sfM(U, V) \ar[r, "\sfS(I{,}\, \tau_{U, V})"] \& \sfS(I, [\tau U, \tau V]) \ar[r, "\pi_{I, \tau U, \tau V}^{-1}", "\cong"'] \& \sfS(I \otimes \tau U, \tau V) \ar[r, "\sfS(l^{-1}{,}\, \tau V)", "\cong"'] \& \sfS(\tau U, \tau V).
		\end{tikzcd}
	\end{center}
\end{cnv}

\begin{prp}[{\cite[\S1.6]{Kel05}}]
	Let $\M$ be an $\sfS$-category. For any $U\in \M$, there is an $\sfS$-functor $\M(U, -) \colon \M \to \cS$ given by $V \mapsto \M(U, V)$ and
	\begin{center}
		\begin{tikzcd}
			\M(U,-)_{V, W} \colon \M(V, W) \ar[r] & \left[\M(U, V), \M(U, W)\right]
		\end{tikzcd}
	\end{center}
	corresponding to the composition $M \colon \M(V, W) \otimes \M(U, V) \to \M(U, W)$ under the adjunction \eqref{eqn: adjunction}, for each $V, W \in \M$. Replacing $\M$ by $\op \M$, yields an $\sfS$-functor $\M(-, W) \colon \op \M \to \cS$, for any $W \in \M$. \qed
\end{prp}

\begin{dfn}
	Let $\M$ be an $\sfS$-category. An $\sfS$-functor $\tau \colon \M \to \cS$ is called \textbf{$\sfS$-representable} if it is $\sfS$-naturally isomorphic to the $\sfS$-functor $\M(U, -)$ \textbf{represented} by some $U \in \M$. Dually, one defines \textbf{$\sfS$-corepresentable} and \textbf{corepresented} $\sfS$-functors $\op \M \to \cS$.
\end{dfn}

\begin{thm}[Weak enriched Yoneda lemma, {\cite[\S1.9]{Kel05}}] \label{thm: Yoneda}
	Let $\sfS$ be a symmetric closed monoidal category with self-enrichment $\cS$. For any $\sfS$-functor $\tau\colon \M \to \cS$ and $U \in \M$, there is a bijection
	\[\sfS\textup{-Nat}(\M(U,-), \tau) \longleftrightarrow \sfS(I, \tau U), \quad \eta = (\eta_V)_{V \in \M} \longleftrightarrow e,\]
	defined by the the following commutative diagrams in $\sfS$:
	\begin{center}
		\begin{tikzcd}[sep={17.5mm,between origins}]
			\M(U, V) \ar[rr, dashed, "\eta_V"] \ar[rd, "\tau_{U, V}"'] && \tau V \\
			& \left[\tau U, \tau V \right] \ar[ru, "-(e)"']
		\end{tikzcd}
		\hspace{1cm}
			\begin{tikzcd}[sep={17.5mm,between origins}]
			I \ar[rr, dashed, "e"] \ar[rd, "j_U"'] && \tau U \\
			& \M(U, U) \ar[ru, "\eta_U"']
		\end{tikzcd}
	\end{center}
	In particular, $\sfS$-natural transformations $\eta \colon \M(U, -) \to \M(U', -)$ correspond to morphisms $f \colon U' \to U$ in $\sfM$ via
	$\eta = \M(f, -) := \left( \M(f, V) \right)_{V \in \M}$.\qed
\end{thm}

\begin{prp}[{\cite[\S1.6]{Kel05}}] \label{prp: Hom-underlying}
	Let $\M$ be an $\sfS$-category. For any, $U, W \in \M$, the functors underlying the $\sfS$-functors $\M(U, -)$ and $\M(-, W)$ are the partial functors of a bifunctor \[\M(-,-) \colon \op \sfM \times \sfM \to \sfS,\] which becomes $[-,-]$ for $\M=\cS$, see \Cref{rmk: self-enrichment}.
	\begin{enumerate}
		\pushQED{\qed}
		\item A morphism $g\colon V \to V'$ in $\sfM$ is assigned to
		\[
			\begin{tikzcd}
				\M(U, g) \colon \M(U, V) \ar[r, "l^{-1}"] & I \otimes \M(U, V) \ar[r, "g \otimes \id"] & \M(V, V') \otimes \M(U, V) \ar[r, "M"] & \M(U, V').
			\end{tikzcd}
		\]
		
		\item A morphism $f\colon V' \to V$ in $\sfM$ is assigned to
		\[
			\begin{tikzcd}
				\M(f, W) \colon \M(V, W) \ar[r, "r^{-1}"] & \M(V, W) \otimes I \ar[r, "\id \otimes f"] & \M(V, W) \otimes \M(V', V) \ar[r, "M"] & \M(V', W).
			\end{tikzcd} \qedhere
		\]
	\end{enumerate}
\end{prp}

\begin{rmk} \label{rmk: comp-j}
	For any object $U \in \M$ of an $\sfS$-category $\M$ and any morphism $f\colon U \to U'$ in the underlying category $\sfM$, we have
	\[\M(U, f) \circ_\sfS j_U = f \circ_\sfM \id_U = f,\]
	due to \Cref{prp: Hom-underlying} and the naturality of $l$: 
	\begin{center}
		\begin{tikzcd}[row sep={15mm,between origins}, column sep={20mm,between origins}]
			M(U, U) \ar[rd, "l^{-1}"] \ar[rrrr, "\M(U{,}\,f)"] &&&& \M(U, U') \ar[ddd, equal] \\
			& I \otimes \M(U, U) \ar[rr, "f \otimes \id"] && \M(U, U') \otimes \M(U, U) \ar[ru, "M"] \ar[d, equal] \\
			& I \otimes I \ar[u, "\id \otimes j_U"] \ar[rr, "f \otimes j_U"] && \M(U, U') \otimes \M(U, U) \ar[rd, "M"'] \\
			I \ar[ru, "l^{-1}"'] \ar[uuu, "j_U"] \ar[rrrr, "f \, \circ_\sfM \, \id_U"] &&&& \M(U, U')
		\end{tikzcd}
	\end{center}
	Dually, we have $\M(f, U') \circ_\sfS j_U = \id_{U'} \circ_\sfM f = f$ in $\sfS$, due to the naturality of $r$.
\end{rmk}

\begin{lem} \label{lem: functor-natural}
	Any $\sfS$-functor $\tau \colon \M \to \N$ yields a natural transformation \[\tau \colon \M(-,-) \to \N(- , -) \circ (\tau \times \tau)\] of bifunctors $\op \sfM \times \sfM \to \sfS$, that is, for any two morphisms $f \colon W \to X$ and $g \colon Y \to Z$ in $\sfM$, there is the following commutative diagram in $\sfS$:
	\begin{center}
		\begin{tikzcd}[row sep={17.5mm,between origins}, column sep={27.5mm,between origins}]
			\M(X, Y) \ar[r, "\M(f{,}\, g)"] \ar[d, "\tau_{X,Y}"] & \M(W, Z) \ar[d, "\tau_{W,Z}"] \\
			\N(\tau X, \tau Y) \ar[r, "\N(\tau f {,}\, \tau g)"] & \N(\tau W, \tau Z)
		\end{tikzcd}
	\end{center}
\end{lem}

\begin{proof} Since $\M(-,-)$ and $\N(-,-)$ are bifunctors, see \Cref{prp: Hom-underlying}, we may assume that $f$ or $g$ is the identity. If $f = \id_X$, the diagram expands to
	\begin{center}
		\begin{tikzcd}[row sep={17.5mm,between origins}, column sep={22.5mm,between origins}]
			\M(X, Y) \ar[rd, "l^{-1}"] \ar[ddd, "\tau_{X,Y}"] \ar[rrrr, "\M(X{,}\, g)"] &&&& \M(X, Z) \ar[ddd, "\tau_{X,Z}"] \\
			& I \otimes \M(X, Y) \ar[rr, "g \otimes \id"] \ar[d, "\id \otimes \tau_{X,Y}"] && \M(Y, Z) \otimes \M(X, Y) \ar[ru, "M"] \ar[d, "\tau_{Y, Z} \otimes \tau_{X,Y}"]  \\
			& I \otimes \N(\tau X, \tau Y) \ar[rr, "\tau g \otimes \id"] && \N(\tau Y, \tau Z) \otimes \N(\tau X, \tau Y) \ar[rd, "M"] \\
			\N(\tau X, \tau Y) \ar[ru, , "l^{-1}"] \ar[rrrr, "\N(\tau X {,}\, \tau g)"] &&&& \N(\tau X, \tau Z).
		\end{tikzcd}
	\end{center}
	The left trapezoid commutes since $l$ is natural, the right one due to \eqref{diag: dfn-functor}, the upper and the lower one commute by \Cref{prp: Hom-underlying}. The interior square commutes, since $\otimes$ is a bifunctor, see \Cref{dfn: underlying-functor}. The second case is dual.
\end{proof}

\begin{prp}[{\cite[\S1.7]{Kel05}}] \label{lem: nat-alt}
	Let $\sigma, \tau \colon \M \to \N$ be $\sfS$-functors. A collection of morphisms $\eta_U \in \sfN(\sigma U, \tau U)$, where $U \in \M$, forms an $\sfS$-natural transformation $\eta \colon \sigma \to \tau$ if and only if the diagram
	\begin{center}
		\begin{tikzcd}[row sep={17.5mm,between origins}, column sep={35mm,between origins}]
			\M(U, V) \ar[r, "\sigma_{U, V}"] \ar[d, "\tau_{U, V}"'] & \N(\sigma U, \sigma V) \ar[d, "\N(\sigma U{,}\, \eta_V)"] \\
			\N(\tau U, \tau V) \ar[r, "\N(\eta_U{,}\, \tau V)"]& \N(\sigma U, \tau V)
		\end{tikzcd} 
	\end{center}
	in $\sfS$ commutes, for all $U, V \in \M$. \qed
\end{prp}

\begin{lem} \label{lem: iso-components}
	An $\sfS$-natural transformation $\eta \colon \sigma \to \tau$ of $\sfS$-functors $\sigma, \tau \colon \M \to \N$ is an $\sfS$-natural isomorphism if and only if $\eta_U \colon \sigma U \to \tau U$ is an isomorphism in $\sfN$, for all $U \in \M$.
\end{lem}

\begin{proof}
	For each $U \in \M$, let $\theta_U \colon \tau U \to \sigma U$ denote the inverse of $\eta_U$ in $\sfN$. For any $U, V \in \M$, consider the following diagram in $\sfS$:
	\begin{center}
		\begin{tikzcd}[row sep={17.5mm,between origins}, column sep={35mm,between origins}]
			\M(U, V) \ar[r, "\tau_{U, V}"] \ar[d,"\sigma_{U, V}"'] & \N(\tau U, \tau V) \ar[d, "\N(\eta_U{,}\, \tau V)"'] \ar[rd, equal] & \\
			\N(\sigma U, \sigma V) \ar[r, "\N(\sigma U{,}\,\eta_V)"] \ar[dr, equal] & \N(\sigma U, \tau V) \ar[r, "\N(\theta_U{,}\, \tau V)"'] \ar[d, "\N(\sigma U{,}\,\theta_V)"] & \N(\tau U, \tau V) \ar[d, "\N(\tau U{,}\, \theta_V)"] \\
			& \N(\sigma U, \sigma V) \ar[r, "\N(\theta_U{,}\,\tau V)"'] & \N(\tau U, \sigma V)
		\end{tikzcd}
	\end{center}
	The upper square commutes due to the $\sfS$-naturality of $\eta$, see \Cref{lem: nat-alt}, the lower one since $\N(-,-)$ is a bifunctor, see \Cref{prp: Hom-underlying}, the triangles by definition of $\theta$. \Cref{lem: nat-alt} applied to the outer hexagon yields the claim.
\end{proof}

\begin{lem} \label{lem: j-extranat}
	Let $\sfS$ be a symmetric closed monoidal category and $\M$ an $\sfS$-category. For any morphism $f \colon X \to Y$ in $\sfM$, there is the following commutative diagram in $\sfS$:
	\begin{center}
		\begin{tikzcd}[row sep={17.5mm,between origins}, column sep={25mm,between origins}]
			I \ar[r, "j_Y"] \ar[d, "j_X"] & \M(Y, Y) \ar[d, "\M(f{,}\, Y)"] \\
			\M(X, X) \ar[r, "\M(X{,}\, f)"] & \M(X, Y)
		\end{tikzcd}
	\end{center} 
\end{lem}

\begin{proof}
	The following diagram commutes due to \Cref{prp: Hom-underlying} and various axioms, see also \Cref{rmk: r=l-I}:
	\begin{center}
		\begin{tikzcd}[row sep={17.5mm,between origins}, column sep={35mm,between origins}]
			\M(Y, Y) \ar[rd, "r^{-1}"] \ar[rrr, "\M(f{,}\, Y)"] &&& \M(X, Y) \\
			 I \ar[rd, "r^{-1}"] \ar[u, "j_Y"] \ar[ddd, equal] & \M(Y, Y) \otimes I \ar[r, "\id \otimes f"] & \M(Y, Y) \otimes \M(X, Y) \ar[ru, "M"] & \M(X, Y) \ar[ddd, equal] \ar[u, equal] \\
			& I \otimes I \ar[r, "\id \otimes f"] \ar[u, "j_Y \otimes \id"] \ar[d, "\id_{I \otimes I} \, = \, c"'] & I \otimes \M(X, Y) \ar[ru, "l"] \ar[u, "j_Y \otimes \id"] \ar[d, "c"'] & \\
			& I \otimes I \ar[r, "f \otimes \id"] \ar[d, "\id \otimes j_X"'] & \M(X, Y) \otimes I \ar[rd, "r"] \ar[d, "\id \otimes j_X"'] & \\
			I \ar[ru, "l^{-1}"] \ar[d, "j_X"'] & I \otimes \M(X, X) \ar[r, "f \otimes \id"] & \M(X, Y) \otimes \M(X, X) \ar[rd, "M"] & \M(X, Y) \ar[d, equal] \\
			M(X, X) \ar[rrr, "\M(X{,}\, f)"] \ar[ru, "l^{-1}"] &&& \M(X, Y)
		\end{tikzcd}
	\end{center}
	Then the commutativity of the outer square yields the claim.
\end{proof}

\section{Bondal--Kapranov for triangulated functors}

In this section, $\sfS$ is a closed additive monoidal category with self-enrichment $\cS$. This endows the underlying category of any $\sfS$-category with a fixed preadditive structure, see \Cref{dfn: underlying-cat}.

\begin{dfn} \label{dfn: enriched-triang} Let $\M$ and $\N$ be $\sfS$-categories with triangulated underlying categories $\sfM$ and $\sfN$.
	\begin{enumerate}
		\item We call an $\sfS$-functor $\tau \colon \M \to \N$ \textbf{triangulated} if its underlying functor $\tau \colon \sfM \to \sfN$ is triangulated.
		
		\item Suppose that $\sfS$ is triangulated. We call $\M$ an \textbf{$\sfS$-triangulated category} if, for any $U \in \M$, the $\sfS$-functors $\M(U,-)\colon \M \to \cS$ and $\M(-, U)\colon \op \M \to \cS$ (co)represented by $U$ are triangulated.
	\end{enumerate}
\end{dfn}

\begin{dfn} \label{dfn: enriched-sub}
	For an $\sfS$-category $\M$, any subset $\ob \U$ of $\ob \M$ gives rise to an \textbf{$\sfS$-subcategory} $\U$ of $\M$ by restricting $\M(-,-)$, $M$, and $j$. Its underlying category $\sfU$ is the full subcategory of the underlying category $\sfM$ with objects $\ob \sfU = \ob \U$.\\
	If $\M$ is $\sfS$-triangulated and $\sfU$ a triangulated subcategory of $\sfM$, we call $\U$ an \textbf{$\sfS$-triangulated subcategory} of $\M$. In particular, $\U$ is an $\sfS$-triangulated category in this case.
\end{dfn}

\begin{dfn}
	Let $\M$ and $\N$ be $\sfS$-categories and $\U$ an $\sfS$-subcategory of $\M$.
	\begin{enumerate}
		\item For any $\sfS$-functor $\tau \colon \M \to \N$, there is a \textbf{restricted} $\sfS$-functor $\tau \vert_\U \colon \U \to \N$ defined by $\tau \vert_\U U := \tau U$ and $\left(\tau \vert_\U\right)_{U, V} := \tau_{U, V}$, for $U, V \in \U$.
		
		\item For any $\sfS$-natural transformation $\eta = (\eta_U)_{U \in \M} \colon \sigma \to \tau$ of $\sfS$-functors $\sigma, \tau \colon \M \to \N$, there is a \textbf{restricted} $\sfS$-natural transformation $\eta\vert_\U := (\eta_U)_{U \in \U} \colon \sigma\vert_\U \to \tau \vert_\U$.
	\end{enumerate}
	
\end{dfn}

\begin{dfn} \label{dfn: SOD}
	Let $\U$ and $\V$ be $\sfS$-triangulated subcategories of $\M$ with underlying triangulated subcategories $\sfU$ and $\sfV$ of $\sfM$. We call the pair $(\U, \V)$ a \textbf{semiorthogonal decomposition} of $\M$ if $\M(\U, \V)=0$ and $\sfM = \sfU \ast \sfV$, that is:
	\begin{enumerate}
		\item The object $\M(U, V)$ of $\sfS$ is zero, for all $U \in \U$ and $V \in \V$.
		
		\item Any $X \in \M$ fits into a distinguished triangle $U \to X \to V$ in $\sfM$ with $U \in \U$ and $V \in \V$.
	\end{enumerate}
\end{dfn}

\begin{rmk} \label{rmk: Hom-nat-iso}
	Consider $\sfS$-subcategories $\U$ and $\V$ of an $\sfS$-triangulated category $\M$ such that $\M(\U, \V)=0$. Due to \Cref{thm: Yoneda} and \Cref{lem: iso-components}, any distinguished triangle $U \xrightarrow{u} X \xrightarrow{v} V$ in $\sfM$ with $U \in \U$ and $V \in \V$ then yields $\sfS$-natural isomorphisms
	\[\M(-, u)\vert_{\U} \colon \U(-, U) \xrightarrow{\cong} \M(-, X)\vert_{\U} \hspace{5mm} \text{ and } \hspace{5mm} \M(v, -)\vert_{\V} \colon \V(V, -) \xrightarrow{\cong} \M(X, -)\vert_{\V}.\]
\end{rmk}

\begin{thm} \label{thm: BK}
	Let $\sfS$ be a symmetric closed additive monoidal and triangulated category.\footnote{We do not assume any compatibility between the triangulation and the monoidal structure.} Consider an $\sfS$-triangulated category $\M$ with a semiorthogonal decomposition $(\U, \V)$. Suppose that all triangulated $\sfS$-functors $\U \to \cS$ and $\V \to \cS$ are $\sfS$-representable. Then any triangulated $\sfS$-functor $\M \to \cS$ is $\sfS$-representable.
\end{thm}

\begin{proof} Let $\tau \colon \M \to \cS$ be a triangulated $\sfS$-functor. We proceed in several steps:
	\begin{enumerate}[wide, label=(\arabic*)]
		\item \label{thm: BK-start} By assumption, there are objects $\tilde U, U' \in \U$ and $\tilde V \in \V$, and $\sfS$-natural isomorphisms
		\begin{align*}
			\eta^\U & =(\eta^\U_U)_{U \in \U}\colon \U(\tilde U, -) \xrightarrow{\cong} \tau\vert_\U, \\
			\eta^\V & =(\eta^\V_V)_{V \in \V}\colon \V(\tilde V, -) \xrightarrow{\cong} \tau\vert_\V, \\
			\theta & =(\theta_U)_{U \in \M} \colon \U(U', -) \xrightarrow{\cong} \M(\tilde V, -)\vert_\U.
		\end{align*}
		By \Cref{thm: Yoneda}, these correspond to $e_{\U} \in \sfS(I, \tau \tilde U)$, $e_{\V} \in \sfS(I, \tau \tilde V)$, and $v \in \sfS(I, \M(\tilde V, U'))=\sfM(\tilde V, U')$ via the following commutative diagrams in $\sfS$:
		\begin{gather} \label{diag: def-e_U}
			\begin{tikzcd}[sep={17.5mm,between origins}, ampersand replacement=\&]
				I \ar[rr, dashed, "e_\U"] \ar[rd, "j_{\tilde U}"'] \&\& \tau \tilde U \\
				\& \M(\tilde U, \tilde U) \ar[ru, "\eta^\U_{\tilde U}"']
			\end{tikzcd}
			\hspace{1cm}
			\begin{tikzcd}[sep={17.5mm,between origins}, ampersand replacement=\&]
				\M(\tilde U, U) \ar[rr, "\eta^\U_U"] \ar[rd, "\tau_{\tilde U, U}"'] \&\& \tau U \\
				\& \left[\tau \tilde U, \tau U \right] \ar[ru, "-(e_\U)"']
			\end{tikzcd}
		\end{gather}
			
		\begin{gather} \label{diag: def-e_V}
			\begin{tikzcd}[sep={17.5mm,between origins}, ampersand replacement=\&]
				I \ar[rr, dashed, "e_\V"] \ar[rd, "j_{\tilde V}"'] \&\& \tau \tilde V \\
				\& \M(\tilde V, \tilde V) \ar[ru, "\eta^\V_{\tilde V}"']
			\end{tikzcd}
			\hspace{1cm}
			\begin{tikzcd}[sep={17.5mm,between origins}, ampersand replacement=\&]
				\M(\tilde V, V) \ar[rr, "\eta^\V_V"] \ar[rd, "\tau_{\tilde V, V}"'] \&\& \tau V \\
				\& \left[\tau \tilde V, \tau V \right] \ar[ru, "-(e_\V)"']
			\end{tikzcd}
		\end{gather}
		
		\begin{gather} \label{diag: def-v}
			\begin{tikzcd}[sep={17.5mm,between origins}, ampersand replacement=\&]
				I \ar[rr, dashed, "v"] \ar[rd, "j_{U'}"'] \&\& \M(\tilde V, U') \\
				\& \M(U', U') \ar[ru, "\theta_{U'}"']
			\end{tikzcd}
		\end{gather}
		
		\item \label{thm: BK-def-u} Define $u \in \sfS(I, \U(\tilde U, U')) = \sfM(\tilde U, U')$ by the left-hand trapezoids in the following diagram in $\sfS$:
		\begin{align} \label{diag: def-u}
			\begin{tikzcd}[row sep={17.5mm,between origins}, column sep={15mm,between origins}, ampersand replacement=\&]
				I \ar[rd, "j_{\tilde V}"] \ar[rrr, equal] \ar[dd, "e_\V"'] \&\&\& I \ar[d, dashed, "u"'] \ar[rrr, equal] \&\&\& I \ar[ld, "j_{\tilde U}"'] \ar[dd, "e_\U"] \\
				\& \V(\tilde V, \tilde V) \ar[rr, dashed] \ar[ld, "\eta^\V_{\tilde V}", "\cong"'] \&\& \U(\tilde U, U') \ar[d, "\eta^\U_{U'}"', "\cong"] \&\& \U(\tilde U, \tilde U) \ar[ll, "\U(\tilde U{,} \, u)"']\ar[rd, "\eta^\U_{\tilde U}"', "\cong"] \\
				\tau \tilde V \ar[rrr, "\tau v"] \&\&\& \tau U' \&\&\& \tau \tilde U \ar[lll, "\tau u"']
			\end{tikzcd}
		\end{align}
		Then the right-hand trapezoids commute due to \Cref{rmk: comp-j,rmk: underlying-nat}. The commutative triangles are from \eqref{diag: def-e_U} and \eqref{diag: def-e_V}.
		
		\item \label{thm: BK-hom-cart} Complete the morphism defined by $u$ and $v$ to a distinguished triangle in $\sfM$ as follows:
		\begin{align} \label{diag: total-triangle}
			\begin{tikzcd}[cramped, ampersand replacement=\&]
				\tilde X \ar{r}{\begin{pmatrix} p \\ -q \end{pmatrix}} \& \tilde U \oplus \tilde V \ar{r}{\begin{pmatrix} u & v \end{pmatrix}} \& U'
			\end{tikzcd}
		\end{align}
		It defines a homotopy cartesian square, which fits into a morphism
		\begin{align} \label{diag: hom-cart}
			\begin{tikzcd}[sep={17.5mm,between origins}, ampersand replacement=\&]
				U'' \ar[r, dashed] \ar[d, equal] \& \tilde X \ar[r, "q"] \ar[d, "p"] \ar[rd, phantom, "\square"] \& \tilde V \ar[d, "v"] \\
				U'' \ar[r] \& \tilde U \ar[r, "u"] \& U'
			\end{tikzcd}
		\end{align}
		of distinguished triangles in $\sfM$, where $U'' \in \U$, see the dual of {\cite[Lem.~1.4.4]{Nee01}}. In the sequel, we establish an $\sfS$-natural isomorphism $\M(\tilde X, -) \cong \tau$.
		
		\item \label{thm: BK-M-v-iso} The restricted $\sfS$-natural transformation, see \Cref{thm: Yoneda},
		\[\M(v, -)\vert_{\U} \colon \U(U', -) \xrightarrow{\cong} \M(\tilde V, -) \vert_{\U}\]
		is an isomorphism: For each $U \in \U$, consider the following diagram in $\sfS$: 
		\begin{gather} \label{diag: nat-theta}
			\begin{tikzcd}[row sep={17.5mm,between origins}, column sep={45mm,between origins}, ampersand replacement=\&]
				\M(U', U) \ar[r, "\U(U'{,}\, -)_{U',U}"] \ar[dd, "\M(\tilde V{,}\, -)_{U',U}"'] \& \left[ \M(U', U'), \M(U', U) \right] \ar[r, "{[j_{U'}{,}\, \id]}"] \ar[d, "{[\id{,}\, \theta_U]}"] \& \left[ I, \M(U', U) \right] \ar[dd, "{[\id{,}\, \theta_U]}"] \\
				\& \left[ \M(U', U'), \M(\tilde V, U) \right] \ar[rd, "{[j_{U'}{,}\, \id]}"] \& \\
				\left[ \M(\tilde V, U'), \M(\tilde V, U) \right] \ar[ru, "{[\theta_{U'}{,}\, \id]}"] \ar[rr, "{[v{,}\,\id]}"] \&\& \left[ I, \M(\tilde V, U) \right]
			\end{tikzcd}
		\end{gather}
		The left trapezoid commutes due to the naturality of $\theta$, see \Cref{lem: nat-alt}. Since the internal hom $[-,-]$ is a bifunctor, see \Cref{prp: int-hom-bifunctor}, the right trapezoid and the triangle commute, see \eqref{diag: def-v}. Now, consider the following diagram in $\sfS$:
		\begin{center}
			\begin{tikzcd}[row sep={12.5mm,between origins}, column sep={20mm,between origins}]
				\M(U', U) \ar[rrdd, "r^{-1}"'] \ar[dddd, equal] \ar[rrrr, equal] &&&& \M(U', U) \ar[dddd, "\theta_U"] \\
				&&& \M(U', U) \otimes \M(U', U') \ar[ru, "M"] \\
				&&\M(U', U) \otimes I \ar[ru, "\id \otimes j_{U'}"'] \ar[rd, "\id \otimes v"]\\
				&&& \M(U', U) \otimes \M(\tilde V, U') \ar[rd, "M"'] \\
				\M(U', U) \ar[rruu, "r^{-1}"] \ar[rrrr,"\M(v{,}\, U)"] &&&& \M(\tilde V, U)
			\end{tikzcd}
		\end{center}
		The commutative triangle on the right hand side is obtained by applying the adjunction isomorphism \eqref{eqn: adjunction} to the two outer composites in \eqref{diag: nat-theta}, see \eqref{eqn: rule-binat} and \eqref{eqn: rule-mid-arg}. The upper triangle commutes due to an axiom, the lower one by \Cref{prp: Hom-underlying}. This shows that $\M(v, U)=\theta_U$ is an isomorphism in $\sfS$, see \ref{thm: BK-start}. The claim follows with \Cref{lem: iso-components}.
		
		\item \label{thm: BK-M-p-q-isos} The restricted $\sfS$-natural transformations, see \Cref{thm: Yoneda},
		\begin{align}
			\M(p, -)\vert_{\U} \colon \U(\tilde U, -) \xrightarrow{\cong} \M(\tilde X, -)\vert_{\U} \hspace{2.5mm} \textup{ and } \hspace{2.5mm}
			\M(q, -)\vert_{\V}\colon \V(\tilde V, -) \xrightarrow{\cong} \M(\tilde X, -)\vert_{\V}
		\end{align}
		 are isomorphisms: For $\M(q, -)\vert_{\V}$, this is immediate from \Cref{rmk: Hom-nat-iso} applied to the upper triangle in \eqref{diag: hom-cart}. For $\M(p, -)\vert_{\U}$, it follows if, for any $U \in \U$, the component $\M(p, U)$ is an isomorphism in $\sfS$, see \Cref{lem: iso-components}. Consider the morphism
		\begin{center}
			\begin{tikzcd}[row sep={17.5mm,between origins}, column sep={35mm,between origins}]
				\M(U', U) \ar[r, "\M(u{,}\, U)"] \ar[d, "\M(v{,}\,U)"'] & \M(\tilde U, U) \ar[r] \ar[d, "\M(p{,}\,U)"'] & \M(U'', U) \ar[d, equal] \\
				\M(\tilde V, U) \ar[r, "\M(q{,}\, U)"] & \M(\tilde X, U) \ar[r] & \M(U'', U)
			\end{tikzcd}
		\end{center}
		of distinguished triangles in $\sfS$, obtained by applying $\M(-, U)$ to \eqref{diag: hom-cart}. Since $\M(v{,}\,U)$ is an isomorphism by \ref{thm: BK-M-v-iso}, then so is $\M(p, U)$.

		\item \label{thm: BK-res} Combining \ref{thm: BK-start} and \ref{thm: BK-M-p-q-isos}, we obtain the following $\sfS$-natural isomorphisms:
		\begin{align} \label{diag: rho-U}
			\begin{tikzcd}[sep={17.5mm,between origins}, ampersand replacement=\&]
				\M(\tilde X, -)\vert_\U \ar[rr, dashed, "\rho^\U", "\cong"'] \&\& \tau\vert_\U \\
				\& \U(\tilde U, -) \ar[ru, "\eta^\U"', "\cong"] \ar[lu, "\M(p{,}\, -)\vert_{\U}", "\cong"']
			\end{tikzcd}
		\end{align}
		\begin{align} \label{diag: rho-V}
			\begin{tikzcd}[sep={17.5mm,between origins}, ampersand replacement=\&]
				\M(\tilde X, -)\vert_\V \ar[rr, dashed, "\rho^\V", "\cong"'] \& \& \tau\vert_\V \\
				\& \V(\tilde V, -) \ar[ru, "\eta^\V"', "\cong"] \ar[lu, "\M(q{,}\, -)\vert_{\V}", "\cong"']
			\end{tikzcd}
		\end{align}
		
		\item \label{thm: BK-prep} To prepare for the next step, we establish the equality \[\tau v \circ \rho^\V_{\tilde V} = \rho^\U_{U'} \circ \M(\tilde X, v).\] Consider the following diagram in $\sfS$, using that $\tau u \circ e_\U = \tau v \circ e_\V$, see \eqref{diag: def-u}:
		\begin{gather} \label{diag: prep-help}
			\begin{tikzcd}[row sep={17.5mm,between origins}, column sep={14mm,between origins}, ampersand replacement=\&]
				\M(\tilde V, \tilde V) \ar[rd, "\tau_{\tilde V, \tilde V}"] \ar[rrr, "\M(\tilde V{,}\, v)"] \ar[dd, "\eta^\V_{\tilde V}"'] \&\&\& \M(\tilde V, U') \ar[d, "\tau_{\tilde V, U'}"'] \&\& \M(U', U') \ar[d, "\tau_{U', U'}"] \ar[ll, "\M(v{,}\, U')"'] \ar[rrr, "\M(u{,}\, U')"] \&\&\& \M(\tilde U, U') \ar[ld, "\tau_{\tilde U, U'}"'] \ar[dd, "\eta^\U_{U'}"] \\
				\&\left[\tau \tilde V, \tau \tilde V \right] \ar[ld, "-(e_\V)"] \ar[rr, "{[} \tau \tilde V {,}\, \tau v {]}"] \&\& \left[\tau \tilde V, \tau U' \right] \ar[d, "-(e_\V)"'] \&\& \left[\tau U', \tau U' \right] \ar[d, "-(\tau u \circ e_\U)", "-(\tau v \circ e_\V)"'] \ar[ll, "{[} \tau v {,}\, \tau U' {]}"'] \ar[rr, "{[} \tau u {,}\, \tau U' {]}"] \&\& \left[\tau \tilde U, \tau U' \right] \ar[rd, "-(e_\U)"'] \\
				\tau \tilde V \ar[rrr, "\tau v"] \&\&\& \tau U' \ar[rr, equal] \&\& \tau U' \ar[rrr, equal] \&\&\& \tau U'
			\end{tikzcd}
		\end{gather}
		The square and the two trapezoids in the upper row commute due to \Cref{lem: functor-natural}. In the lower row, the left trapezoid commutes by \Cref{rmk: evaluation}.\ref{rmk: evaluation-nat}, the square and the right trapezoid by \Cref{rmk: evaluation}.\ref{rmk: evaluation-comp}. The commutative triangles are from \eqref{diag: def-e_U} and \eqref{diag: def-e_V}. Using \ref{thm: BK-M-v-iso}, the contour of \eqref{diag: prep-help} then becomes the commutative jagged shape in the following diagram in $\sfS$:
		\begin{gather} \label{diag: prep-comm}
			\begin{tikzcd}[sep={17.5mm,between origins}, ampersand replacement=\&]
				\M(\tilde X, \tilde V) \ar[rrrr, "\M(\tilde X{,}\, v)"] \ar[ddd, "\rho^\V_{\tilde V}"]\&\&\&\& \M(\tilde X, U') \ar[rrr, equal] \&\&\& \M(\tilde X, U') \ar[ddd, "\rho^\U_{U'}"] \\
				\& \M(\tilde V, \tilde V) \ar[lu, "\M(q{,}\, \tilde V)"', "\cong"] \ar[rrdd, "\eta^\V_{\tilde V}"] \ar[rr, "\M(\tilde V{,}\, v)"] \&\& \M(\tilde V, U') \ar[ru, "\M(q{,}\, U')"] \&\& \M(\tilde U, U') \ar[lu, "\M(p{,}\, U')"', "\cong"] \ar[rrdd, "\eta^\U_{U'}"] \\
				\&\&\&\& \M(U', U') \ar[lu, <-, "\M(v{,}\, U')^{-1}", "\cong"'] \ar[ru, "\M(u{,}\, U')"'] \\
				\tau \tilde V \ar[rrr, equal] \&\&\& \tau \tilde V \ar[rrrr, "\tau v"] \&\&\&\& \tau U'
			\end{tikzcd}
		\end{gather}
		In addition, the two large triangles commute due to \eqref{diag: rho-U} and \eqref{diag: rho-V}. The commutative square in the middle is obtained by applying $\M(-, U')$ to the homotopy cartesian square from \eqref{diag: hom-cart} and inverting $\M(v, U')$. The trapezoid commutes since $\M(-,-)$ is a bifunctor, see \Cref{prp: Hom-underlying}. The claimed equality then follows from the resulting commutativity of \eqref{diag: prep-comm}.

		\item To extend $\rho^\U \colon \M(\tilde X, -) \vert_{\U} \to \tau\vert_{\U}$ and $\rho^\V \colon \M(\tilde X, -) \vert_{\V} \to \tau\vert_{\V}$ from \ref{thm: BK-res} to an $\sfS$-natural transformation $\M(\tilde X, -) \to \tau$ on all of $\M$, we apply the triangulated functors $\M(\tilde X, -)$ and $\tau$ to the distinguished triangle \eqref{diag: total-triangle} in $\sfM$. Using the naturality of $\rho^\U$, see \Cref{rmk: underlying-nat}, and the equality from \ref{thm: BK-prep}, we obtain an isomorphism
		\begin{gather} \label{diag: def-eta-comp}
			\begin{tikzcd}[row sep={20mm,between origins}, column sep={50mm,between origins}, ampersand replacement=\&]
				\M(\tilde X, \tilde X) \ar{r}{\begin{pmatrix} \M(\tilde X, p) \\ -\M(\tilde X, q) \end{pmatrix}} \ar[d, dashed, "\cong", "\eta_{\tilde X}"'] \& \M(\tilde X, \tilde U) \oplus \M(\tilde X, \tilde V) \ar{r}{\begin{pmatrix} \M(\tilde X, u) & \M(\tilde X, v) \end{pmatrix}} \ar[d, color=white, "\textcolor{black}{\cong}"] \& \M(\tilde X, U') \ar[d, "\rho^\U_{U'}"', "\cong"] \\
				\tau\tilde X \ar{r}{\begin{pmatrix} \tau p \\ -\tau q \end{pmatrix}} \&  \tau \tilde U \oplus \tau \tilde V\ar[<-]{u}{\begin{pmatrix}
						\rho^\U_{\tilde U} & 0 \\ 0 & \rho^\V_{\tilde V}
				\end{pmatrix}} \ar{r}{\begin{pmatrix} \tau u & \tau v \end{pmatrix}} \& \tau U'
			\end{tikzcd}
		\end{gather}
		of distinguished triangles in $\sfS$. By \Cref{thm: Yoneda}, $e := \eta_{\tilde X} \circ j_{\tilde X} \in \sfS(I, \tau \tilde X)$ corresponds to an $\sfS$-natural transformation $\eta =(\eta_U)_{U \in \M} \colon \M(\tilde X, -) \to \tau$ via the following commutative diagrams in $\sfS$:
		
		\begin{gather} \label{diag: def-eta}
			\begin{tikzcd}[sep={17.5mm,between origins}, ampersand replacement=\&]
				I \ar[rr, dashed, "e"] \ar[rd, "j_{\tilde X}"'] \&\& \tau \tilde X \\
				\& \M(\tilde X, \tilde X) \ar[ru, "\eta_{\tilde X}"']
			\end{tikzcd}
			\hspace{1cm}
			\begin{tikzcd}[sep={17.5mm,between origins}, ampersand replacement=\&]
			\M(\tilde X, U) \ar[rr, dashed, "\eta_U"] \ar[rd, "\tau_{\tilde X, U}"'] \&\& \tau U \\
				\& \left[\tau \tilde X, \tau U \right] \ar[ru, "-(e)"']
			\end{tikzcd}
		\end{gather}
		
		\item \label{thm: BK-res-iso} The restrictions $\eta\vert_\U$ and $\eta\vert_\V$ are isomorphisms: For $U \in \U$, consider the following diagram in $\sfS$:
		\begin{center}
			\begin{tikzcd}[row sep={17.5mm,between origins}, column sep={25mm,between origins}]
				I \ar[rd, "j_{\tilde X}"] \ar[rrr, "j_{\tilde U}"] \ar[dd, "e"'] &&& \M(\tilde U, \tilde U) \ar[ld, "\M(p{,}\, \tilde U)"'] \ar[dd, "\eta^\U_{\tilde U}"] \\
				& \M(\tilde X, \tilde X) \ar[r, "\M(\tilde X{,}\, p)"] \ar[ld, "\eta_{\tilde X}"'] & \M(\tilde X, \tilde U) \ar[rd, "\rho^\U_{\tilde U}"] \\
				\tau \tilde X \ar[rrr, "\tau p"] &&& \tau \tilde U
			\end{tikzcd}
		\end{center}
		The trapezoids commute by \Cref{lem: j-extranat} and \eqref{diag: def-eta-comp}, the triangles by \eqref{diag: def-eta} and \eqref{diag: rho-U}.
		Using \eqref{diag: def-e_U}, it follows that \[\tau p \circ e = \eta^\U_{\tilde U} \circ j_{\tilde U} = e_\U.\]
		Then \Cref{rmk: evaluation}.\ref{rmk: evaluation-comp} yields the commutative lower trapezoid in the following diagram in $\sfS$:
		\begin{center}
			\begin{tikzcd}[row sep={17.5mm,between origins}, column sep={25mm,between origins}]
				\M(\tilde U, U) \ar[rrr, "\M(p{,}\,U)", "\cong"'] \ar[rd, "\tau_{\tilde U, U}"] \ar[dd, "\eta^\U_U"', "\cong"] &&& \M(\tilde X, U) \ar[ld, "\tau_{\tilde X, U}"] \ar[dd, "\eta_U"]\\
				&\left[ \tau \tilde U, \tau U \right] \ar[r, "{[}\tau p{,}\, \tau U{]}"] \ar[ld, "-(e_\U)"'] & \left[ \tau \tilde X, \tau U \right] \ar[rd, "-(e)"]\\
				\tau U \ar[rrr, equal] &&& \tau U
			\end{tikzcd}
		\end{center}
		The upper trapezoid commutes due to \Cref{lem: functor-natural}, the triangles by \eqref{diag: def-e_U} and \eqref{diag: def-eta}. It follows that \[\eta \vert_\U = \eta^\U \circ \left(\M(p, -)\vert_\U\right)^{-1} = \rho^\U\] is an isomorphism, see \eqref{diag: rho-U} and \ref{thm: BK-res}. The proof for $\eta \vert_\V = \rho^\V$ is analogous.
		
		\item It remains to see that $\eta_X$ is an isomorphism in $\sfS$ for any $X \in \T$, see \Cref{lem: iso-components}. Consider a distinguished triangle \begin{tikzcd}[cramped, sep=small] U \ar[r] & X \ar[r] & V \end{tikzcd} in $\sfM$, where $U \in \U$ and $V \in \V$. By the naturality of $\eta$, applying the triangulated functors $\M(\tilde X, -)$ and $\tau$ yields a morphism of distinguished triangles
		\begin{center}
			\begin{tikzcd}[row sep={17.5mm,between origins}, column sep={25mm,between origins}]
				\M(\tilde X, U) \ar[r] \ar[d, "\eta_U"', "\cong"] & \M(\tilde X, X) \ar[r] \ar[d, "\eta_X"'] & \M(\tilde X, V) \ar[d, "\eta_V"', "\cong"] \\
				\tau U \ar[r] & \tau X \ar[r] & \tau V
			\end{tikzcd}
		\end{center}
		in $\sfS$ with isomorphisms as indicated due to \ref{thm: BK-res-iso}. Then $\eta_X$ is an isomorphism in $\sfS$. \qedhere
	\end{enumerate}
\end{proof}

With an additional hypothesis, a modification of the proof of \Cref{thm: BK} yields our main result:

\begin{thm} \label{thm: BK-alt-hyp}
	Let $\sfS$ be a symmetric closed additive monoidal and triangulated category. Consider an $\sfS$-triangulated category $\M$ with semiorthogonal decompositions $(\W, \U)$ and $(\U, \V)$. Then a triangulated $\sfS$-functor $\tau \colon \M \to \cS$ is $\sfS$-representable if $\tau \vert_{\U}$ and $\tau \vert_{\V}$ are so.
\end{thm}

\begin{proof}
	The existence of $\theta$ in step \ref{thm: BK-start} of the proof of \Cref{thm: BK} can be deduced from the additional semiorthogonal decomposition $(\W, \U)$: Any distinguished triangle $\begin{tikzcd}[cramped, sep=small] W \ar[r] & \tilde V \ar[r, "v"] & U' \end{tikzcd}$ with $W \in \W$ and $U' \in \U$ serves to define $v$. Then \[\theta := \M(v, -) \vert_{\U}\colon \U(U', -) \to \M(\tilde V, -)\vert_\U\] is an $\sfS$-natural isomorphism due to \Cref{rmk: Hom-nat-iso}.
\end{proof}

\begin{rmk}
	Under the assumptions of \Cref{thm: BK-alt-hyp}, $\theta$ can be eliminated entirely from the proof of \Cref{thm: BK}: In step \ref{thm: BK-def-u}, $v$ from the proof of \Cref{thm: BK-alt-hyp} is used to define $u$. The argument from step \ref{thm: BK-hom-cart} then applies twice to extend the diagram \eqref{diag: hom-cart} symmetrically:
	\begin{align}
		\begin{tikzcd}[sep={17.5mm,between origins}, ampersand replacement=\&]
			\& W \ar[r, equal] \ar[d, dashed] \& W \ar[d] \\
			U'' \ar[r] \ar[d, equal] \& \tilde X \ar[r, "q"] \ar[d, "p"] \ar[rd, phantom, "\square"] \& \tilde V \ar[d, "v"] \\
			U'' \ar[r] \& \tilde U \ar[r, "u"] \& U'
		\end{tikzcd}
	\end{align}
	In step \ref{thm: BK-M-p-q-isos}, the argument for $\M(q, -)\vert_{\V}$ now also covers $\M(p, -)\vert_{\U}$, making step \ref{thm: BK-M-v-iso} obsolete. From step \ref{thm: BK-res} onward, the proof remains unchanged.
\end{rmk}


\printbibliography

\end{document}